\documentclass[11pt]{article}
\usepackage{mathrsfs}
\usepackage{latexsym,lineno}
\usepackage{epsfig}
\usepackage{color}
\usepackage{amsmath}\usepackage{fleqn}\usepackage{verbatim}\usepackage{epsf}
\usepackage{amsthm}\usepackage{graphicx, float}\usepackage{graphicx}
\usepackage{amsfonts}\usepackage{amssymb}\usepackage{graphpap}
\usepackage{epic}\usepackage{curves}

\newcommand{\be}{\begin{equation}}
\newcommand{\ee}{\end{equation}}
\newcommand{\benum}{\begin{enumerate}}
\newcommand{\eenum}{\end{enumerate}}
\newcommand{\bit}{\begin{itemize}}
\newcommand{\eit}{\end{itemize}}
\newtheorem{thom}{Theorem}[section]
\newtheorem{lemma}{Lemma}[section]

\newtheorem{prop}{Proposition}[section]

\newtheorem{defn}{Definition}[section]

\usepackage{graphicx}

\begin{document}
\def\s{\subseteq}
\def\n{\noindent}
\def\se{\setminus}
\def\dia{\diamondsuit}
\def\la{\langle}
\def\ra{\rangle}


\title{On extremal multiplicative Zagreb indices of trees with given domination number}

\author{Shaohui Wang$^a$, ~Chunxiang Wang$^b$, Jia-Bao Liu$^{c,}$\footnote{Corresponding author: S. Wang (e-mail: shaohuiwang@yahoo.com; swang@adelphi.edu), C. Wang(e-mail: wcxiang@mailccnu.edu.cn), J.B. Liu (e-mail:  liujiabaoad@163.com).}   \\
\small\emph {a. Department of Mathematics and Computer Science, Adelphi University, Garden City, NY 11550, USA}\\
\small\emph {b. School of Mathematics and Statistics, Central China Normal University, Wuhan,
430079, P.R. China}\\
\small\emph {c. School of Mathematics and Physics, Anhui Jianzhu
University, Hefei 230601, P.R. China}}
\date{}
\maketitle

\begin{abstract}

For a graph $G$, the first multiplicative Zagreb index $\prod_1$ is equal to the product of squares of the vertex degrees, and the second multiplicative Zagreb index $\prod_2$ is equal to the product of the products of degrees of pairs of adjacent vertices. The (mutiplicative) Zagreb indices have been the focus of considerable research in computational chemistry dating back to Gutman and  Trinajsti\'c in 1972.
  In this paper, we
explore the mutiplicative Zagreb indices in terms of arbitrary
domination number. The sharp upper and lower bounds of
$\prod_1(G)$ and $\prod_2(G)$ are given. In addition, the
corresponding extreme graphs are charaterized.

\end{abstract}

\section{Introduction}
All graphs considered in this paper are simple, connected graphs.
Let $G = (V, E)$ be such a graph, where $V = V (G)$ is its vertex
set and $E = E(G)$ is its edge set.  For $u\in V(G)$, $G-u$ is an induced subgraph of $V(G) - \{u\}$ in $G$.
A graph $G$ that has $n$
vertices and $n-1$ edges is called a tree. As usual, by $P_n$ and
$K_{1,n-1}$ denote the path and the star on $n$ vertices,
respectively.

Molecular descriptors could be helpful for QSAR/QSPR studies and
for the descriptive purposes of biological and chemical
properties, such as melting and boiling points, toxicity,
physico-chemical, and biological
properties~\cite{wei1,wei2,Gutman1996,Liu2015,LiuP2015,LiuPX2015,0001,0002}.
 One of the first topological molecular descriptors
is so-called Zagreb indices~\cite{Gutman1972}, which are auxiliary
quantities in an approximated formulae for the total
$\pi$-electron energy of conjugated molecules. Many results of
the applications on Zagreb indices were explored
in~\cite{Gutman2014}. Recently, there are hundreds of articles
investigated Zagreb indices in the area of chemistry and
mathematics~\cite{Hu2005,Li2008,shi2015,BF2014,SM2014,Xu2014,WangJ2015}.

The degree-based graph invariants $M_1(G)$ and
$M_2(G)$~\cite{Gutman1972} are defined as
 \be
\label{1}
M_1(G) = \sum_{u \in V(G)} d(u)^2,~
 M_2(G) = \sum_{uv \in E(G)} d(u)d(v).
\ee

In 2010, Todeschini et al.~\cite{RT20101,RT20102} presented the
following multiplicative variants of molecular structure
descriptors:
 \be \label{3}
\prod_1(G) = \prod_{u \in V(G)} d(u)^2,~
\prod_2(G) = \prod_{uv \in E(G)} d(u)d(v). \ee
By the recursive process, we see that $\prod_2(G) = \prod_{uv \in E(G)} d(u)d(v) = \prod_{u \in V(G)}
d(u)^{d(u)}$.

Recently, multiplicative Zagreb indices attracted extensive
attention in physics, chemistry, graph theory, etc.
 Xu and Hua~\cite{Xu20102} proposed a unified approach to
characterize extremal (maximal and minimal) trees, unicyclic
graphs and bicyclic graphs with respect to multiplicative Zagreb
indices, respectively.
Iranmanesh et al.~\cite{Iranmanesh20102} investigated these
indices the first and the second multiplicative Zagreb indices for
a class of dendrimers.
Liu and Zhang [14] introduced several sharp upper bounds for
$\pi_1$-index and $\pi_2$-index in terms of graph parameters
including the order, size and radius~\cite{Liuz20102}.
 Wang and Wei~\cite{Wang2015} studied these indices in
$k$-trees and extremal $k$-trees were characterized.
Ramin Kazemi~\cite{Ramin2016} obtained the bounds for
the moments and the probability generating function of these
indices in a randomly chosen molecular graph with tree structure
of order $n$.
Bojana Borovi\'canin et al.~\cite{Borov2016} presented upper
bounds on Zagreb indices of trees in terms of domination number. Also, a lower bound for the
first Zagreb index of trees with a given domination number is
determined and the extremal trees are characterized as well.

Motivated by the above results, in this paper we further
investigate the multiplicative Zagreb indices of trees in terms of
domination number. This enriches and extends some earlier results
obtained by Bojana Borovi\'canin et al.~\cite{Borov2016}.

The rest of the paper is organized as follows. In Section 2, we
provide some useful lemmas and preliminaries.
The lower bounds of first
multiplicative Zagreb index and upper bounds of second Zagreb
index on trees of given domination number in Section 3. The upper
bounds of first multiplicative Zagreb index and lower bounds of
second mutiplicative Zagreb index on trees of given domination
number in Section 4. 

\section{Preliminaries}

In this section, we provide some propositions and lemmas which are critical in the following proofs. 

\begin{prop}
The function $f(x) = \frac{x}{x+m}$ with $m > 0$ is increasing in $\mathbb{R}$.
\end{prop}

\begin{prop}
The function $g(x) = \frac{x^x}{(x+m)^{x+m}}$ with $m > 0$ is decreasing in $\mathbb{R}$.
\end{prop}

\begin{lemma}
(\cite{Borov2016}) Let $T$ be a tree with n vertices and
domination number $\frac{n+3}{3} \leq \gamma \leq \frac{n}{2}$.
If the value of $max  \{|d(u) - d(v)|, u,v \in V(T) \}$ is as
small as possible $($say $d(u) \in \{1, 2, 3\})$, and $n_1$ is the
number of vertices of degree 1, then $n_1 \geq 3\gamma - n$, where
the inequality is strict if there exists a vertex with two pendent
neighbors.
\end{lemma}

If $\Delta(G) = 1$, $G$ is an edge by the assumption that $G$ is connected, and corresponding results are trivial.
We will conisder  graphs $G$ with $\Delta(G) \geq 2$ below.

\section{Lower bounds of first multiplicative Zagreb index and upper bounds of second Zagreb index on the trees}

In this section, we provide  sharp lower bounds of first multiplicative Zagreb index and upper bounds of second Zagreb index on trees with $n$ vertices and domination number $\gamma$. The corresponding extreme graphs are given in the following definition.

\begin{defn}
Denoted by $T_{n, \gamma}$ the tree obtained from a star graph
$K_{1, n- \gamma}$ by attaching a pendant edge to its $\gamma - 1$
pendant vertices.
\end{defn}

Let $\mathcal{T}_{n, \gamma}$ be the class of trees $T_{n, \gamma}$.
Note that $\gamma = 1$ if and only if $T \cong K_{1, n-1}$.
If $\Delta = n-\gamma$ in a tree of order $n$ and domination number $\gamma$, then $T \cong T_{n,\gamma}$. The first multiplicative Zagreb index and second multiplicative Zagreb index of $T_{n, \gamma}$ can be calculated routinely below.

\begin{prop} Let $T_{n, \gamma}$ $\in \mathcal{T}_{n, \gamma}$. Then
\begin{eqnarray}
\nonumber
\prod_1(T_{n, \gamma}) = 4^{\gamma -1}(n-\gamma)^2, ~
\prod_2(T_{n, \gamma}) = 4^{\gamma - 1} (n-\gamma)^{n-\gamma}.
\end{eqnarray}
\end{prop}

\subsection{Lower bounds of $\prod_1(G)$ on trees with domination number $\gamma$}
\begin{thom}
\label{lb}
Let $G$ be a tree with $n$ vertices and domination number $\gamma$.
 Then
\begin{eqnarray}
 \nonumber \prod_1(G)\geq 4^{\gamma-1}(n-\gamma)^2, ~\text{the equality holds if and only if} ~ G \cong T_{n, \gamma}.
\end{eqnarray}
\end{thom}

\begin{proof} We first consider $\Delta(G) = 2$.
Since $n \geq 2$ and $G$ is a connected graph, then $T \cong P_n$.
For $n = 2, 3$ or $4$, $G \cong T_{2,1} (\equiv P_2), T_{3,1}(\equiv P_3)$ or $T_{4,2}(\equiv P_4)$, respevtively. By the routine caculations of $\prod_1(G)$, we have that the equality of Theorem \ref{lb}  holds. If $n \geq 5$, then the inequality of Theorem \ref{lb} holds by direct calculations of $\prod_1(G)$. Next we will consider trees with $\Delta(G) \geq 3$.

Set $P_{d+1}:= v_1v_2\dots v_{d+1}$ to be a longest
path in $G$, where $d$ is the diameter of $G$. We have
$d(v_1)=d(v_{d+1})=1.$
Let $D$ be arbitrary minimal dominating set of $G$ such that  $|D|=\gamma$.  Then
 $\triangle\leq n-\gamma$. As one can routinely calculated for $n \leq 5$, Theorem \ref{lb} is true. Now we will prove it by the induction on $n \geq 6$.  Assume that Theorem \ref{lb} holds for $|G| = n-1$, and we will show the case of $|G| = n$. There are two possible cases below.

{\bf \small Case 1.} Suppose that $\gamma(G - v_1) = \gamma(G).$ Then $v_2$ is not in the choosed domination set. By the concept of $\prod_1(G)$ and $d(v_1) = 1$, we obtain that
\begin{eqnarray}
\label{8}
\nonumber \prod_1(G) &&=\prod_{v \in V(G)} d(v)^2=(\prod_{v \in{V(G) \setminus \{v_1,v_2\}}}d(v)^2) \cdot d(v_1)^2d(v_2)^2
\\
\nonumber &&= (\prod_{v \in V(G) \setminus \{v_1, v_2\}}d(v)^2) (d(v_2)-1)^2 \frac{d(v_2)^2}{(d(v_2)-1)^2} d(v_1)^2
\\
 &&=\prod_1(G  - v_1)\cdot
\frac{d(v_2)^2}{(d(v_2)-1)^2}\cdot 1.
\end{eqnarray}

By the induction hypothesis, we have
\begin{eqnarray}
\label{9}
\nonumber \prod_1(G)
&&\geq4^{\gamma-1}(n-1-\gamma)^2\cdot \frac{d(v_2)^2}{d(v_2-1)^2}
\\  \nonumber
&&\geq4^{\gamma-1}(n-\gamma)^2\cdot \frac{(n-1-\gamma)^2}{(n-\gamma)^2}\cdot \frac{d(v_2)^2}{d(v_2-1)^2}\\
\nonumber &&= 4^{\gamma-1}(n-\gamma)^2\cdot (\frac{\frac{n-1-\gamma}{n-\gamma}}{\frac{d(v_2)-1}{d(v_2)}})^2
\\
 &&\geq4^{\gamma-1}(n-\gamma)^2.
\end{eqnarray}
Thus, Theorem \ref{lb} is proved. Based on the induction hypothesis,   equality  (\ref{9}) holds if and only if $d(v_1) = 1$ and $d(v_2) = n- \gamma$, that is, $G \cong T_{n, \gamma}$.

{\bf \small Case 2.} Suppose that $\gamma(G - v_1) = \gamma(G) - 1.$ Then $v_2$ is in every domination set and $d(v_2)=2.$ By the induction hypothesis and the concept of $\prod_1(G)$, we have that
\begin{eqnarray}
\label{10}
\nonumber \prod_1(G) &&=\prod_1(G -v_1)\cdot\frac{d(v_2)^2}{(d(v_2)-1)^2}
\\\nonumber
&&\geq4^{\gamma-2}(n-\gamma)^2\cdot (\frac{d(v_2)}{d(v_2)-1})^2
\\  \nonumber
 &&= \frac{4^{\gamma -1}}{4} \cdot (n- \gamma)^2 \cdot (\frac{2}{2-1})^2
\\
 &&=4^{\gamma-1}(n-\gamma)^2.
\end{eqnarray}
Thus, Theorem \ref{lb} is true. Based on the induction hypothesis, the relation (\ref{10})  holds if and only if $d(v_1) = 1, d(v_2) = 2$ and $G\setminus \{v_1\} \cong T_{n-1, \gamma -1}$, that is, $G \cong T_{n, \gamma}$.
Therefore, Theorem \ref{lb} is proved.
\end{proof}

\subsection{Upper bounds of $\prod_2(G)$ on trees with domination number $\gamma$}
\begin{thom}
\label{up}
Let $G$ be a tree with domination number $\gamma$. Then
\begin{eqnarray}
\label{11} \nonumber
\prod_2(G)\leq 4^{\gamma-1}(n-\gamma)^{n-\gamma}, ~\text{the equality holds if and only if}~G\cong T_{n, \gamma}.
\end{eqnarray}
\end{thom}

\begin{proof}
We consider $\Delta(G) = 2$  firstly. As  $G$ is a connected graph with $n \geq 2$,  $T \cong P_n$.
If $n = 2, 3$ or $4$, then $G \cong T_{2,1} (\equiv P_2), T_{3,1}(\equiv P_3)$ or $T_{4,2}(\equiv P_4)$, respevtively. By the direct caculations of $ \prod_2(G)$, we have that the equality of Theorem \ref{up}  holds. If $n \geq 5$, then the inequality of Theorem \ref{up} holds by routine calculations of $\prod_2(G)$. Next we will consider trees with $\Delta(G) \geq 3$.

 Set $P_{d+1}:= v_1v_2\dots v_{d+1}$ to be a longest
path in $G$, where $d$ is the diameter of $G$. We have
$d(v_1)=d(v_{d+1})=1.$
Let $D$ be any minimal domination set of $G$ such that  $|D|=\gamma$.  Then
 $\triangle\leq n-\gamma$. As one can routinely calculated for $n \leq 5$, Theorem \ref{up} is true and we focus on $n > 6$. Now we will prove it by the induction on $n$.  We suppose that Theorem \ref{up} is true for $|G| = n-1$, and consider the case of $|G| = n$. Here we have two seperate cases.

{\bf \small Case 1.} Suppose that $\gamma(G - v_1) = \gamma(G).$ Then $v_2$ is not in the choosed domination set. By the definition  of $\prod_2(G)$ and $d(v_1) =1$, we obtain that
\begin{eqnarray}
\label{12}
\nonumber \prod_2(G) &&=\prod_{v \in V(G)} d(v)^{d(v)}=(\prod_{v \in{V(G) \setminus \{v_1,v_2\}}}d(v)^{d(v)}) \cdot d(v_1)^{d(v_1)}d(v_2)^{d(v_2)}
\\
\nonumber &&= (\prod_{v \in V(G) \setminus \{v_1, v_2\}}d(v)^{d(v)}) (d(v_2)-1)^{d(v_2)-1} \frac{d(v_2)^{d(v_2)}}{(d(v_2)-1)^{d(v_2)-1}} d(v_1)^{d(v_1)}
\\
 &&=\prod_1(G - v_1)\cdot
\frac{d(v_2)^{d(v_2)}}{(d(v_2)-1)^{d(v_2) - 1}}\cdot 1.
\end{eqnarray}

By the induction on  $|G| = n-1$, we have
\begin{eqnarray}
\label{13}
\nonumber \prod_2(G)
&&\leq4^{\gamma-1}(n-1-\gamma)^{n-1-\gamma}\cdot \frac{d(v_2)^{d(v_2)}}{(d(v_2)-1)^{d(v_2)-1}}
\\  \nonumber
&&= 4^{\gamma-1}(n-\gamma)^{n-\gamma}\cdot \frac{(n-1-\gamma)^{n-1-\gamma}}{(n-\gamma)^{n-\gamma}}\cdot \frac{d(v_2)^{d(v_2)}}{(d(v_2)-1)^{d(v_2)-1}}\\
\nonumber &&= 4^{\gamma-1}(n-\gamma)^{n-\gamma}\cdot\frac{ \frac{(n-1-\gamma)^{n-1-\gamma}}{(n-\gamma)^{n-\gamma}}}{ \frac{(d(v_2)-1)^{d(v_2)-1}}{d(v_2)^{d(v_2)}}}
\\
 &&\leq4^{\gamma-1}(n-\gamma)^{n-\gamma}.
\end{eqnarray}
Thus, Theorem \ref{up} is proved. Based on the induction hypothesis, (\ref{12}) and (\ref{13}) hold if and only if $d(v_1) = 1$ and $d(v_2) = n- \gamma$, that is, $G \cong T_{n, \gamma}$.

{\bf \small Case 2.} Suppose that $\gamma(G - v_1) = \gamma(G) - 1.$ Then $v_2$ is in every domination set and $d(v_2)=2.$ By the induction hypothesis and the definition of $\prod_2(G)$, we have that
\begin{eqnarray}
\label{14}
\nonumber \prod_2(G)
&&\leq4^{\gamma-2}(n-\gamma)^{n-\gamma}\cdot \frac{d(v_2)^{d(v_2)}}{(d(v_2)-1)^{d(v_2)-1}}
\\  \nonumber
&&= \frac{4^{\gamma-1}}{4}(n-\gamma)^{n-\gamma}\cdot  \frac{d(v_2)^{d(v_2)}}{(d(v_2)-1)^{d(v_2)-1}}\\
\nonumber &&= \frac{4^{\gamma-1}}{4}(n-\gamma)^{n-\gamma}\cdot \frac{2^2}{1^1}
\\
 &&= 4^{\gamma-1}(n-\gamma)^{n-\gamma}.
\end{eqnarray}

Thus, Theorem \ref{up} is true.  Based on the induction hypothesis,  equality (\ref{14}) holds if and only if $d(v_1) = 1, d(v_2) = 2$ and $G\setminus \{v_1\} \cong T_{n-1, \gamma -1}$, that is, $G \cong T_{n, \gamma}$.
Therefore, Theorem \ref{up} is proved.
\end{proof}

\section{Upper bounds of first multiplicative Zagreb index and lower bounds of second mutiplicative Zagreb index on the trees}

In this section, we study the upper bounds of first multiplicative Zagreb index and lower bounds of second mutiplicative Zagreb index on  trees of $n$ vertices and domination number $\gamma$. Here we first introduce some facts which are useful in the proofs of these results.

It is known that $1 \leq \gamma \leq \frac{n}{2}$, and  $\gamma(G) = 1$ if and only if $G \cong K_{1,n-1}$. Note that the path $P_n$ is a unique tree of order $n$ and $\gamma(G) = \lceil{\frac{n}{3}}\rceil$ such that $\prod_1(G)$ is maximal or $\prod_2(G)$ is minimal. In this section, we seperately consider two cases of $\gamma \leq \frac{n}{3}$ and $\frac{n}{3} < \gamma \leq \frac{n}{2}$  below (Here we keep similar notations of \cite{Borov2016,VSD}).

Let $D$ be arbitrary minimal dominating set of a tree $G$ with $n$ vertices and domination number $\gamma$, and $\overline{D} = V(T) \setminus D$. Thus, $|D| = \gamma$ and $|\overline{D}| = n- \gamma$.
Denote by $l$, $k$ or $p$ the number of edges $uv \in E(G)$ such that $u \in D$ and $v \in \overline{D}$,  $u \in D$ and $v \in {D}$, and $u \in \overline{D}$ and $v \in \overline{D}$, respectively. Since $G$ is a tree, then
\be
\label{15}
k+l+p = n-1.
\ee

By the structures of $D$ and $\overline{D}$, we have
\be
\label{16}
\sum_{u \in D} d(u) = l +2k,
\ee
\be
\label{17}
\sum_{v \in \overline{D}}d(v) = l+2p,
\ee
 \be
\label{18}
\prod_1(G) = (\prod_{u \in D} d(u)^2)(\prod_{v \in \overline{D}} d(v)^2), ~  \prod_2(G) = (\prod_{u \in D} d(u)^{ d(u)})(\prod_{v \in \overline{D}} d(v)^{d(v)}).
\ee

Based on the concept of domination number, $l \geq n - \gamma$ and (\ref{15}) yield that $k+p \leq \gamma -1$. Then
\be
\label{19}
|k-p| \leq \gamma -1.
\ee

Note that (by \cite{Iranmanesh20102,Wang2015}) the product of $d(u)^2$ (or $d(u)^{d(u)}$, respectively) with $u \in D$ necessarily attain the maximum (or minimum, respectively) if degrees $d(u)$ differ at most one among each other, i.e., if $d(u) \in \{ \lceil{\frac{l+2k}{\gamma}}\rceil, \lfloor{\frac{l+2k}{\gamma}}\rfloor\}$ for $u \in D$.
Similarly, the product of $d(v)^2$ (or $d(v)^{d(v)}$, respectively) with $v \in \overline{D}$ necessarily attain the maximum (or minimum, respectively) if  $d(v) \in
\{ \lceil{\frac{l+2p}{n-\gamma}}\rceil, \lfloor{\frac{l+2p}{n-\gamma}}\rfloor
\}$ for $v \in \overline{D}$.

Let $l+2k = q\gamma + r$, where $0 \leq r \leq \gamma - 1$, $q = \lfloor{\frac{l+2k}{\gamma}}\rfloor$ and $r = l+2k - \gamma\lfloor{\frac{l+2k}{\gamma}}\rfloor$. Based on  the relation $(\ref{16})$,  $\prod_{u \in  D} d(u)^2$ (or $\prod_{u \in D} d(u)^{d(u)}$, respevtively) is maximal (or minimal, respectively) if $D$ has $r$ vertices of degree $q+1$ and $\gamma - r$ vertices of degree $q$. Combining with the relation $l = n - 1-k-p$, we obtain
\begin{eqnarray}
\label{20}
\nonumber \prod_{u \in D} d(u)^2
&& \leq (q+1)^{2r} \cdot q^{2(\gamma -r)}
\\  \nonumber
&&= (\lfloor{\frac{l+2k}{\gamma}}\rfloor +1)^{2(l+2k - \gamma \lfloor{\frac{l+2k}{\gamma}}\rfloor)} \cdot  \lfloor{\frac{l+2k}{\gamma}}\rfloor ^ {2(\gamma - l - 2k + \gamma \lfloor{\frac{l+2k}{\gamma}}\rfloor)}
\\
\nonumber
 &&= (\lfloor{\frac{n-1+(k-p)}{\gamma}}\rfloor+1)^{2(n-1+(k-p) - \gamma \lfloor{\frac{n-1+(k-p)}{\gamma}}\rfloor)}
 \\
&&~~~~ \lfloor{\frac{n-1+(k-p)}{\gamma}}\rfloor ^{2(\gamma - n+1 -(k-p) + \gamma \lfloor{\frac{n-1+(k-p)}{\gamma}}\rfloor)},
\end{eqnarray}
\begin{eqnarray}
\label{201}
\nonumber \prod_{u \in D} d(u)^{d(u)}
&& \geq (q+1)^{(q+1)r} \cdot q^{q(\gamma -r)}
\\  \nonumber
&&= (\lfloor{\frac{l+2k}{\gamma}}\rfloor +1)^{(\lfloor{\frac{l+2k}{\gamma}}\rfloor +1)(l+2k - \gamma \lfloor{\frac{l+2k}{\gamma}}\rfloor)} \cdot  \lfloor{\frac{l+2k}{\gamma}}\rfloor ^ {\lfloor{\frac{l+2k}{\gamma}}\rfloor(\gamma - l - 2k + \gamma \lfloor{\frac{l+2k}{\gamma}}\rfloor)}
\\
\nonumber &&= (\lfloor{\frac{n-1+(k-p)}{\gamma}}\rfloor+1)^{(\lfloor{\frac{n-1+(k-p)}{\gamma}}\rfloor+1)(n-1+(k-p) - \gamma \lfloor{\frac{n-1+(k-p)}{\gamma}}\rfloor)}
 \\
&&~~~~ \lfloor{\frac{n-1+(k-p)}{\gamma}}\rfloor ^{\lfloor{\frac{n-1+(k-p)}{\gamma}}\rfloor(\gamma - n+1 -(k-p) + \gamma \lfloor{\frac{n-1+(k-p)}{\gamma}}\rfloor)}.
\end{eqnarray}

Next, let $l+2p = Q(n- \gamma) +R$, where $Q = \lfloor{\frac{l+2p}{n- \gamma}}\rfloor$ and $R= l +2p - (n- \gamma)\lfloor{\frac{l+2p}{n- \gamma}}\rfloor$. Similarly, based on (\ref{17}), $\prod_{v \in \overline{D}} d(v)^2$ ($\prod_{v \in \overline{D}} d(v)^{d(v)},$ respectively) is maximal (or minimal, respectively) if $\overline{D}$ has $R$ vertices of degree $Q+1$ and $n-\gamma -R$ vertices of degree $Q$. Combining with the relation $l = n - 1-k-p$, we have
\begin{eqnarray}
\label{21}
\nonumber \prod_{v \in \overline{D}} d(v)^2
&& \leq (Q+1)^{2R} \cdot Q^{2(n- \gamma -R)}
\\  \nonumber
&&= (\lfloor{\frac{l+2p}{n-\gamma}}\rfloor +1)^{2(l+2p - (n - \gamma) \lfloor{\frac{l+2p}{n-\gamma}}\rfloor)} \cdot  \lfloor{\frac{l+2p}{n-\gamma}}\rfloor ^ {2(n-\gamma - l - 2p + (n - \gamma) \lfloor{\frac{l+2p}{n - \gamma}}\rfloor)}
\\  \nonumber
 &&= (\lfloor{\frac{n-1+(p-k)}{n - \gamma}}\rfloor+1)^{2(n-1+(p-k) - (n - \gamma) \lfloor{\frac{n-1+(p-k)}{n-\gamma}}\rfloor)}
 \\ 
&&~~~~ \lfloor{\frac{n-1+(p-k)}{n-\gamma}}\rfloor ^{2(-\gamma +1 -(p-k) + (n-\gamma) \lfloor{\frac{n-1+(p-k)}{n-\gamma}}\rfloor)},
\end{eqnarray}
\begin{eqnarray}
\label{211}
\nonumber \prod_{v \in \overline{D}} d(v)^{d(v)}
&& \geq (Q+1)^{(Q+1)R} \cdot Q^{Q(n- \gamma -R)}
\\  \nonumber
&&= (\lfloor{\frac{l+2p}{n-\gamma}}\rfloor +1)^{(\lfloor{\frac{l+2p}{n-\gamma}}\rfloor +1)(l+2p - (n - \gamma) \lfloor{\frac{l+2p}{n-\gamma}}\rfloor)} \cdot  \lfloor{\frac{l+2p}{n-\gamma}}\rfloor ^ {\lfloor{\frac{l+2p}{n-\gamma}}\rfloor(n-\gamma - l - 2p + (n - \gamma) \lfloor{\frac{l+2p}{n - \gamma}}\rfloor)}
\\ \nonumber
 &&= (\lfloor{\frac{n-1+(p-k)}{n - \gamma}}\rfloor+1)^{(\lfloor{\frac{n-1+(p-k)}{n - \gamma}}\rfloor+1)(n-1+(p-k) - (n - \gamma) \lfloor{\frac{n-1+(p-k)}{n-\gamma}}\rfloor)}
 \\  
&&~~~~ \lfloor{\frac{n-1+(p-k)}{n-\gamma}}\rfloor ^{\lfloor{\frac{n-1+(p-k)}{n-\gamma}}\rfloor (-\gamma +1 -(p-k) + (n-\gamma) \lfloor{\frac{n-1+(p-k)}{n-\gamma}}\rfloor)}.
\end{eqnarray}

Togethering with (\ref{18}), (\ref{20}),(\ref{201}),(\ref{21})
 and (\ref{211}),  we obtain that
\begin{eqnarray}
\label{22}
\nonumber \prod_1(G)
&&=  (\prod_{u \in D} d(u)^2) (\prod_{v \in \overline{D}} d(v)^2)
\\  \nonumber
 && \leq (\lfloor{\frac{n-1+(k-p)}{\gamma}}\rfloor+1)^{2(n-1+(k-p) - \gamma \lfloor{\frac{n-1+(k-p)}{\gamma}}\rfloor)}
 \\  \nonumber
&&~~~~ \lfloor{\frac{n-1+(k-p)}{\gamma}}\rfloor ^{2(\gamma - n+1 -(k-p) + \gamma \lfloor{\frac{n-1+(k-p)}{\gamma}}\rfloor)}\\
&&~~~~\nonumber
(\lfloor{\frac{n-1+(p-k)}{n - \gamma}}\rfloor+1)^{2(n-1+(p-k) - (n - \gamma) \lfloor{\frac{n-1+(p-k)}{n-\gamma}}\rfloor)}
 \\
&&~~~~ \lfloor{\frac{n-1+(p-k)}{n-\gamma}}\rfloor ^{2(-\gamma +1 -(p-k) + (n-\gamma) \lfloor{\frac{n-1+(p-k)}{n-\gamma}}\rfloor)},
\end{eqnarray}
\begin{eqnarray}
\label{221}
\nonumber \prod_2(G)
&&=  (\prod_{u \in D} d(u)^{d(u)}) (\prod_{v \in \overline{D}} d(v)^{d(v)})
\\  \nonumber
 && \geq (\lfloor{\frac{n-1+(k-p)}{\gamma}}\rfloor+1)^{(\lfloor{\frac{n-1+(k-p)}{\gamma}}\rfloor+1)(n-1+(k-p) - \gamma \lfloor{\frac{n-1+(k-p)}{\gamma}}\rfloor)}
 \\  \nonumber
&&~~~~ \lfloor{\frac{n-1+(k-p)}{\gamma}}\rfloor ^{\lfloor{\frac{n-1+(k-p)}{\gamma}}\rfloor(\gamma - n+1 -(k-p) + \gamma \lfloor{\frac{n-1+(k-p)}{\gamma}}\rfloor)}
\\&&~~~~\nonumber
(\lfloor{\frac{n-1+(p-k)}{n - \gamma}}\rfloor+1)^{(\lfloor{\frac{n-1+(p-k)}{n - \gamma}}\rfloor+1)(n-1+(p-k) - (n - \gamma) \lfloor{\frac{n-1+(p-k)}{n-\gamma}}\rfloor)}
 \\
&&~~~~ \lfloor{\frac{n-1+(p-k)}{n-\gamma}}\rfloor ^{\lfloor{\frac{n-1+(p-k)}{n-\gamma}}\rfloor (-\gamma +1 -(p-k) + (n-\gamma) \lfloor{\frac{n-1+(p-k)}{n-\gamma}}\rfloor)}.
\end{eqnarray}

Since $n, \gamma$ are fixed, then we consider the  right-side hands of the relations (\ref{22}) and (\ref{221}) as   functions about $k-p$, say $f(k-p)$ and $g(k-p)$ with $|k-p| \leq \gamma -1$, respectively.  It is enough to find the maximal value of $f(k-p)$ and the minimal value of $g(k-p)$ below.

\subsection{Upper bounds of $\prod_1(G)$ on  trees of domination number $\gamma$}

Let $G$ be a tree of $n$ vertices and domination number $\gamma$. In order to find the maximal values of $\prod_1(G)$,  we first consider the case of $1 \leq \gamma \leq \frac{n}{3}$. The corresponding extreme graphs are given in the following definition.

\begin{defn}
$\mathcal{D}$$(n,\gamma)$ is a set of $n$-vertex trees $D_{n,\gamma}$ with domination number $\gamma$ such that $D_{n,\gamma}$ consists of the stars of orders $\lfloor{\frac{n}{\gamma}}\rfloor$ and $\lceil{\frac{n}{\gamma}}\rceil$ with exact $\gamma - 1$ pairs of adjacent leaves in neighboring stars. (See an example of Figure 1.)
\end{defn}

\begin{figure}[htbp]
    \centering
    \includegraphics[width=4in]{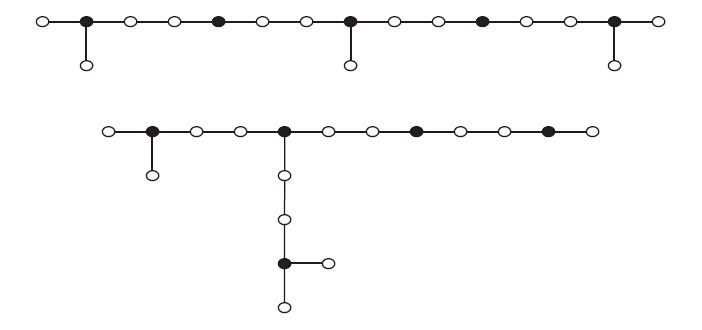}
    \caption{ Two non-isomorphic trees of $D_{18, 5}$ \cite{Borov2016}.}
    \label{fig: te}
\end{figure}

Note that the degrees of $D_{n,\gamma}$ are $n-3\gamma +2$ vertices of degree 1, $2\gamma -2$ vertices of degree 2, $2\gamma -n +\gamma\lfloor{\frac{n-\gamma}{\gamma}}\rfloor$ vertices of degree $\lfloor{\frac{n-\gamma}{\gamma}}\rfloor$  and $n-\gamma - \gamma \lfloor{\frac{n-\gamma}{\gamma}}\rfloor$ vertices of degree $\lfloor{\frac{n-\gamma}{\gamma}}\rfloor +1$, where $\lfloor{\frac{n-\gamma}{\gamma}}\rfloor$ may equal to 2. Figure 1  is an example of $D_{n,\gamma}$ such that $n=18, \gamma = 5$.

\begin{thom}
\label{up1}
Let $G$ be a tree of $n$ vertices and domination number  $1 \leq \gamma \leq \frac{n}{3}$. Then
 $$\prod_1(G) \leq 4^{2\gamma-2} \lfloor{\frac{n-\gamma}{\gamma}}\rfloor^{2(2\gamma - n + \gamma \lfloor{\frac{n-\gamma}{\gamma}}\rfloor)} (\lfloor{\frac{n-\gamma}{\gamma}}\rfloor +1)^{2(n-\gamma - \gamma \lfloor{\frac{n-\gamma}{\gamma}}\rfloor)},$$
 where the equality holds if and only if $G \cong D_{n,\gamma}$.
\end{thom}

\begin{proof} We proceed on $f(k-p)$ and determine its maximum.
If $n=3$, then $T \cong P_3$,  $\gamma = 1$ and Theorem \ref{up1} is true. If $n > 3$, as $\gamma \leq \frac{n}{3}$, then $n-\gamma \geq \frac{2n}{3} ~\text{and } \frac{\gamma -1}{n - \gamma} \leq \frac{n-3}{2n} \leq \frac{1}{2}.$ Thus, \begin{eqnarray}
\label{=1}
\lfloor\frac{n-1+p-k}{n-\gamma}\rfloor = 1.
\end{eqnarray}

Here we consider $q= \lfloor{\frac{l+2k}{\gamma}}\rfloor=\lfloor{\frac{n-1+k-p}{\gamma}}\rfloor$. Since $\frac{n-1+k-p}{\gamma} \geq \frac{n-1-\gamma +1}{ \gamma} = \frac{n-\gamma}{\gamma} \geq \frac{2n/3}{n/3} = 2$,  then
\begin{eqnarray}
\label{q>2}
q \geq 2.
\end{eqnarray}  By combing with above relations,
\begin{eqnarray}
\nonumber f(k-p) &&= (q+1)^{2(n-1+(k-p) - \gamma q)} q^{2(1-n+\gamma - (k-p) + \gamma q)} 2^{2(n-1+(p-k) - (n-\gamma) \cdot 1)} \cdot 1
\\  \nonumber
  && = (q+1)^{2(k-p)} q^{2(- (k-p) )} 2^{2(-(k-p))} \cdot (q+1)^{2(n-1- \gamma q)} q^{2(1-n+\gamma + \gamma q)} 2^{2(n-1 - (n- \gamma))}
\\  \nonumber
  && = (\frac{1/2}{q/(q+1)})^{2(k-p)}
\cdot (q+1)^{2(n-1- \gamma q)} q^{2(1-n+\gamma + \gamma q)} 2^{2(\gamma -1)}.
\end{eqnarray}
As $q \geq 2, n, \gamma$ are fixed, by Proposition 1, $f(k-p)$ is a decreasing function with the variable of $k-p$.
Since $|k-p| \leq \gamma -1$, then there are two cases below.

{\small\bf Case 1:} $0 \leq k-p \leq \gamma -1$.

Note that $\frac{n-1}{\gamma} \leq \frac{n-1+k-p}{\gamma} \leq \frac{n-1}{\gamma} + \frac{\gamma -1}{\gamma} < \frac{n-1}{\gamma} +1.$ Then we have

\be
\label{23}
\lfloor{\frac{n-1+k-p}{\gamma}}\rfloor = \lfloor{\frac{n-1}{\gamma}}\rfloor, ~\text{for}~ 0 \leq k-p \leq \gamma \lfloor{\frac{n-1}{\gamma}}\rfloor+ \gamma - n,
\ee
and
\be
\label{24}
\lfloor{\frac{n-1+k-p}{\gamma}}\rfloor = \lfloor{\frac{n-1}{\gamma}}\rfloor +1, ~\text{for}~ \gamma \lfloor{\frac{n-1}{\gamma}}\rfloor+ \gamma - n +1 \leq k-p \leq \gamma - 1.
\ee
As an addendum, the relation (\ref{23}) holds if $n = \gamma \lfloor{\frac{n-1}{\gamma}}\rfloor +1$.

 By (\ref{23}) and (\ref{24}), $k-p$ falls in two intervals and the maximum values of $f(k-p)$ arrived at either $k-p=0$ or $k-p = \gamma \lfloor{\frac{n-1}{\gamma}}\rfloor + \gamma -n +1$.
In order to find which one is bigger, we need to compare $f(\gamma \lfloor{\frac{n-1}{\gamma}}\rfloor + \gamma -n +1)$ and $f(0)$. Note that
$\gamma \lfloor{\frac{n-1}{\gamma}}\rfloor + \gamma -n +1 \geq (n - 1 -\gamma) + \gamma - n +1 = 0$.

\begin{eqnarray}
\label{25}
\nonumber \frac{f(\gamma \lfloor{\frac{n-1}{\gamma}}\rfloor + \gamma -n +1)}{f(0)} &&= \frac{(1/2)^{2(\gamma \lfloor{\frac{n-1}{\gamma}}\rfloor + \gamma -n +1)}}{(q/(q+1))^{2 \cdot 0}}
\\  \nonumber
  && = (1/2)^{2(\gamma \lfloor{\frac{n-1}{\gamma}}\rfloor + \gamma -n +1)}
\\  \nonumber
  && \leq (1/2)^{2 \cdot 0}
\\
  &&= 1.
\end{eqnarray}

Thus, $f(0)$ is maximum when $ 0 \leq k - p \leq \gamma -1$. Also,
\begin{eqnarray}
\label{26}
\nonumber f(0) &&= (q+1)^{2(n-1-\gamma -q)} q^{2(1-n+\gamma+\gamma q)} 2^{2(n-1- (n- \gamma))}
\\
  && =(\lfloor{\frac{n-1}{\gamma}}\rfloor+1)^{2(n-1- \gamma \lfloor{\frac{n-1}{\gamma}}\rfloor)} \lfloor{\frac{n-1}{\gamma}}\rfloor  ^2(1-n+\gamma + \gamma \lfloor{\frac{n-1}{\gamma}}\rfloor) 2^{2(\gamma -1)}.
\end{eqnarray}

{\small\bf Case 2:} $- \gamma +1 \leq k-p \leq 0$.

Note that $\frac{n-1}{\gamma} - 1 \leq \frac{n-\gamma}{ \gamma} \leq \frac{n-1+k-p}{\gamma} \leq \frac{n-1}{\gamma}$.
Let $n-1 =  Q\gamma +R$, where $0 \leq R \leq \gamma -1$. For $0 \leq R \leq \gamma -2$, we have

\be
\label{27}
\lfloor{\frac{n-1+k-p}{\gamma}}\rfloor = \lfloor{\frac{n-1}{\gamma}}\rfloor, ~\text{for}~ \gamma \lfloor{\frac{n-1}{\gamma}}\rfloor - n +1 \leq k-p \leq 0,
\ee
and
\be
\label{28}
\lfloor{\frac{n-1+k-p}{\gamma}}\rfloor = \lfloor{\frac{n-1}{\gamma}}\rfloor -1, ~\text{for}~ -\gamma + 1 \leq k-p \leq  \gamma \lfloor{\frac{n-1}{\gamma}}\rfloor - n.
\ee
As an addendum, the relation (\ref{27}) holds if $n = \gamma \lfloor{\frac{n-1}{\gamma}}\rfloor +1$.

Note that $f(k-p)$ is a decreasing function on these two intervals of (\ref{27}) and (\ref{28}), for $0 \leq R \leq \gamma-2$. Thus, $f(k-p)$ arrives at the maximum value for either $k-p = \gamma \lfloor{\frac{n-1}{\gamma}}\rfloor - n +1$ or $k-p = - \gamma +1$ (If $R = \gamma -1$, then $\gamma \lfloor{\frac{n-1}{\gamma}}\rfloor - n +1 = -\gamma +1$).

Note that $\gamma \lfloor{\frac{n-1}{\gamma}}\rfloor - n +1 \geq n-1-\gamma-n+1 \geq -\gamma$, $\frac{n-\gamma}{\gamma} = \lfloor\frac{n-1}{\gamma}\rfloor + \frac{R-\gamma+1}{\gamma}$ and $0 \leq R \leq \gamma -2$.
By combing above relations,
\begin{eqnarray}
\label{29}
\nonumber \frac{f(\gamma \lfloor{\frac{n-1}{\gamma}}\rfloor - n +1)}{f(-\gamma +1)}
  && = (\frac{1/2}{Q/(Q+1)})^{2((\gamma \lfloor{\frac{n-1}{\gamma}}\rfloor - n +1)) - (- \gamma +1))}\\
 && \nonumber = (\frac{1/2}{Q/(Q+1)})^{2((\gamma (\frac{n-\gamma}{\gamma} - \frac{R-\gamma+1}{\gamma}) - n +1)) - (- \gamma +1))}\\
&& \nonumber = (\frac{1/2}{Q/(Q+1)})^{2(\gamma - R -1)}\\
  &&  \leq  (\frac{1/2}{Q/(Q+1)})^{2 \cdot 1} \leq 1.
\end{eqnarray}

Finally, we need to compare these two maximum values $f(-\gamma+1)$ and $f(0)$. Since $-\gamma+1 \leq 0$ and $f(k-p)$ is decreasing, then
 the largest value of $f(k-p)$ is arrived at $k-p = - \gamma +1$ and for $k=0,$ $p= \gamma-1$ and $l = n - \gamma$. Hence, Theorem \ref{up1} is true and the equality holds if and only if $G \in \mathcal{D}_{n,\gamma}$.
\end{proof}

Next we consider the trees of $n$ vertices and domination number $\frac{n}{3} \leq \gamma \leq \frac{n}{2}$. The corresponding extreme graphs are given in the following definition.

\begin{defn}
$\mathcal{L}(n, \gamma)$ is a set of trees $L_{n,\gamma}$ with $n$ vertices and domination number $\gamma$, such that every vertex from $L_{n,\gamma}$ has at most one pendent neighbor, and
\\(i) there exists a minimum dominating set $D$ of $L_{n,\gamma}$ containing $3 \gamma - n -2$ vertices of degree 3, and $2n-4\gamma$ vertices of degree 2, while the set $\overline{D}$ contains $n- 2\gamma +2$ vertices of degree 2 and $3\gamma - n $ pendent vertices, or\\
(ii) there exists a minimum dominating set $D$ of $L_{n,\gamma}$ containing $n-2\gamma$ vertices of degree 2 and $3\gamma -n$ pedent vertices, while the set $\overline{D}$ has
$2n-4\gamma+2$ vertices of degree 2, $3\gamma-n-2$ vertices of degree 3 and any vertices from $\overline{D}$ gas exactly one neighbor in $D$. (See an example of Figure 2.)
\end{defn}

\begin{figure}[htbp]
    \centering
    \includegraphics[width=4in]{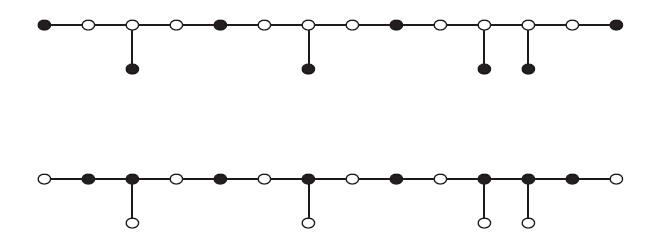}
    \caption{ Two non-isomorphic trees of $L_{28, 5}$ \cite{Borov2016}.}
    \label{fig: te}
\end{figure}

Note that $L_{n,\gamma}$ contains $3\gamma - n - 2$ vertices of degree $3$, $3n-6\gamma+2$ vertices of degree $2$ and $3\gamma-n$ vertices of degree $1$.

\begin{thom}
\label{310}
Let $G$ be a tree of $n$ vertices and domination number  $\frac{n}{3} \leq \gamma \leq \frac{n}{2}$. Then
\begin{eqnarray}
\prod_1(G) \leq
\begin{cases}
\nonumber 4^{n-2}, &\gamma = \lceil{\frac{n}{3}}\rceil,
 \cr 4^{3n-6\gamma+2} 9^{3\gamma - n -2}, &\frac{n+3}{3} \leq \gamma \leq \frac{n}{2}, \end{cases}
\end{eqnarray}
 where the equality holds if and only if $G \cong L_{n,\gamma}$.
\end{thom}

\begin{proof}
We proceed on $f(k-p)$ and determine its maximum.
If $\gamma = \lceil{\frac{n}{3}}\rceil$, $P_n$ is a path. Thus, $\prod_1(G)$ is maximal and Theorem \ref{310} is true. Here we consider the case of $\gamma \geq \frac{n+3}{3}$.

Note that $2\gamma \leq n \leq 3\gamma -3$ yields $\gamma \geq 3$ and $n \geq 6$. Also,
$1 = \frac{n-\gamma}{n-\gamma} \leq \frac{n-1+p-k}{n-\gamma} \leq \frac{n+\gamma - 2}{n-\gamma} = 1+ 2 \frac{\gamma - 1}{\gamma} < 3.$
 Thus,
\be
\label{123}
q_1 = \lfloor{\frac{n-1+p-k}{n-\gamma}}\rfloor = 1 ~\text{or}~ q_1 = \lfloor{\frac{n-1+p-k}{n-\gamma}}\rfloor = 2.
\ee

{\bf \small Case 1:} $q_1 = \lfloor{\frac{n-1+p-k}{n-\gamma}}\rfloor = 1$.

Note that $1  \leq \frac{n-1+p-k}{n-\gamma} < 2$ yields $p - k < n - 2\gamma +1,$ i.e., $$k-p \geq 2\gamma -n.$$
If $2\gamma - n \leq -1,$ then $2 \leq \frac{n-1}{\gamma} \leq 3 \frac{n-1}{n+3} < 3$. Thus, $\lfloor{\frac{n-1}{\gamma}}\rfloor = 2,$ for $2\gamma - n \leq -1$.
Now we first consider $2\gamma -n \leq k-p \leq 0$. By the same ideas of the relations (\ref{27}) and (\ref{28}), we have

\be
\label{32}
\lfloor{\frac{n-1+k-p}{\gamma}}\rfloor = \lfloor{\frac{n-1}{\gamma}}\rfloor = 2, ~\text{for}~ 2\gamma  - n +1 \leq k-p \leq 0,
\ee
and
\be
\label{33}
\lfloor{\frac{n-1+k-p}{\gamma}}\rfloor = \lfloor{\frac{n-1}{\gamma}}\rfloor -1 = 1, ~\text{for}~  k-p =  2\gamma  - n.
\ee

Since $q_1 = \lfloor{\frac{n-1+p-k}{n-\gamma}}\rfloor = 1$, then the relation (\ref{32}) holds only. Otherwise, $P_n$ is a counter-example by $\gamma \geq \frac{n+3}{3}$.

Then
\begin{eqnarray}
\label{34}
\nonumber f(k-p) && = (2+1)^{2(n-1+k-p-2\gamma)} 2^{2(\gamma - n +1 - (k-p) +2\gamma)} (1+1)^{2(n-1 +(p-k) - (n-\gamma))} 1
\\
  &&  =(3/4)^{2(k-p)} 3^{2(n-1-2\gamma)} 2^{2(4\gamma-n)},~\text{for}~ 2\gamma -n+1 \leq k-p \leq 0.
\end{eqnarray}

Next assume that $0 \leq k-p \leq \gamma -1.$ By the same ideas of the relations (\ref{23}) and (\ref{24}), we have

\be
\label{35}
\lfloor{\frac{n-1+k-p}{\gamma}}\rfloor = \lfloor{\frac{n-1}{\gamma}}\rfloor =2, ~\text{for}~ 0 \leq k-p \leq 3\gamma  - n,
\ee
and
\be
\label{36}
\lfloor{\frac{n-1+k-p}{\gamma}}\rfloor = \lfloor{\frac{n-1}{\gamma}}\rfloor +1 = 3, ~\text{for}~ 3\gamma  - n +1 \leq k-p \leq \gamma - 1.
\ee

By  the relations (\ref{22}), (\ref{35}) and (\ref{36}), we have

 \be
\label{37}
f(k-p) =  (3/4)^{2(k-p)} 3^{2(n-1-2\gamma)} 2^{2(4\gamma-n)}, ~\text{for}~ 0 \leq k-p \leq 3\gamma  - n,
\ee
and
\be
\label{38}
f(k-p) =  (2/3)^{2(k-p)}  2^{2(2n-5\gamma -3)} 3^{2(4\gamma - n +1)}, ~\text{for}~ 3\gamma  - n +1 \leq k-p \leq \gamma - 1.
\ee

Togethering with the relations (\ref{34}) and (\ref{37}), we have
\begin{eqnarray}
\label{39}
 f(k-p)   =(3/4)^{2(k-p)} 3^{2(n-1-2\gamma)} 2^{2(4\gamma-n)},~\text{for}~ 2\gamma -n+1 \leq k-p \leq 3\gamma -n.
\end{eqnarray}

By the relations  (\ref{38}) and (\ref{39}), we have  $\frac{f(3\gamma -n)}{f(3\gamma -n+1)} = \frac{16}{9} > 1$.
Since the minimal value $f(3\gamma -n)$ of the relation (\ref{39}) is bigger than the maximum value $f(3\gamma -n+1)$ of the relation (\ref{38}). In order to find the maximum value of $f(k-p)$, we should consider the relation  (\ref{39}) only.

To find the sharp upper bound of $\prod_1(G)$, where $G$ is a tree with  $n$ vertices and domination n   number $\gamma$. It is enough to find the maximum realizable value of $k-p$, such that the corresponding tree exists. We will proceed on these steps below.

 First, note that an extreme tree $G$ with a maximum $\prod_1(G)$ contains vertices of degree $1, 2$ or $3$.
By the above considerations, any minimal dominating set $D$ has $n_3$ verices of degree $3$ and $n_2$ vertices of degree $2$, i.e., $n_2+n_3 = \gamma$. Also, the set $V(G) \setminus D$ has $n_1$ vertices of degree $1$ and $\overline{n_2}$ vertices of degree $2$, i.e., $\overline{n_2} + n_1 = n- \gamma.$

As $n = n_1+n_2+\overline{n_2}+n_3,$ the relation $(1)$ can be written as $n_1+2(n_2+\overline{n_2}) + 3n_3 = 2(n_1+n_2+\overline{n_2}+n_3) - 2$. Thus,
 \begin{eqnarray}
\label{40}
n_3 = n_1 -2.
\end{eqnarray}

Combining with these relations, we have $n_2 - \overline{n_2} = 2\gamma - n +2$. By using  (\ref{40}), the relations  (\ref{16}) and  (\ref{17}) could be $n-1+k-p = 2n_2 +3n_1-6$ and $n-1+p-k = 2\overline{n_2}+n_1$. Thus, $k-p= n_1+2\gamma-n-1$.

 Thus,  the function  (\ref{39}) can be expressed as
  \begin{eqnarray}
\label{41}
f(n_1) =3^{2(n_1-2) 2^{2(n-2n_1+2)}} =  (3/4)^{2n_1}3^{-2}2^{2n+4}, ~\text{for}~ 2 \leq n_1 \leq \gamma+1.
\end{eqnarray}

Now we turn to the case of $2\gamma - n = 0$, i.e., $\gamma = n/2$ if $n$ is even. Then $\lfloor{\frac{n-1}{\gamma}}\rfloor = 1$ and similar to the relations  (\ref{23}) and  (\ref{24}), we have $q = \lfloor{\frac{n-1+k-p}{\gamma}}\rfloor = \lfloor{\frac{n-1}{\gamma}}\rfloor = 1$ for $k-p=0$ and $q = \lfloor{\frac{n-1}{\gamma}}\rfloor +1 =2,$ for $1 \leq k-p \leq \frac{n}{2} -1.$
Recall that at the same time $q_1 = 1$ and, consequently, it has to be $q =2$ (since for $q = 1, T \cong P_n,$ a contadiction, as $\gamma \geq \frac{n+3}{3}$).

By the same method above, we have $f(n_1) =  (3/4)^{2n_1}3^{-2}2^{2n+4},$ for $2 \leq n_1 \leq \frac{n}{2}$.

Thus, we should find the minimal value of $n_1$ such that there exists such trees with n vertices and domination number $\gamma$ with
\begin{eqnarray}
\label{42}
\frac{n+3}{3} \leq \gamma \leq \frac{n}{2}.
\end{eqnarray}
 Note that the vertices from any dominating set $D$ of $G$ have degrees 2 and 3, and the vertices $\overline{D}$ have degrees 1 and 2.

By Lemma 1.1, we have $n_1 \geq 3\gamma -n.$ Then the maximal possible value of $f(n_1)$ is achieved for $n_1 = 3\gamma-n$,
 i.e., $k-p = 5\gamma -2n-1$ and $f(5\gamma-2n-1) = 4^{3n-6\gamma+2} 9^{3\gamma - n -2}$. In addition, the extreme graphs of achieving the equality in Theorem \ref{310} belong to $\mathcal{L}(n,\gamma)$.

{\bf \small Case 2:} $q_1 = \lfloor{\frac{n-1+p-k}{n-\gamma}}\rfloor = 2.$

Note that $2 \leq \frac{n-1+p-k}{n-\gamma} < 3$ yields that $p-k \geq n-2\gamma +1$, that is $k-p \leq 2\gamma -n -1 < 0$. Also, $1 \leq \frac{n-1-\gamma+1}{\gamma} \leq \frac{n-1+k-p}{\gamma} \leq \frac{n-1+2\gamma-n-1}{\gamma} = 2 \frac{\gamma -1}{\gamma} < 2$ implies that $q = \lfloor{\frac{n-1+k-p}{\gamma}}\rfloor = 1.$

For $p-k = n - 2\gamma+1$ and any minimal dominating set $D$, $\frac{n-1+p-k}{n-\gamma} = 2$ and all vertices of $\overline{D}$ are degree of 2. If the vertices of $D$ are degree of $1$  or  $2$, then $T \cong P_n$ and it is contradicted with the assumption.
By $\gamma \geq \frac{n+3}{3}$, we have $p-k \geq n-2\gamma+2,$ i.e., $k-p \leq 2\gamma - n -2$.

Thus, we have $$f(k-p) = (4/3)^{2(k-p)} 2^{2(3n-4\gamma+1)} 3^{2(2\gamma - n+1)},~  \text{for}~ -\gamma +1 \leq k-p \leq 2\gamma - n - 2.$$ Next we need to determine the minimum realization of $k-p$ such that the related tree exists. Here we will proceed in by the same method of previous case.
Let $n_1, n_2$ be the number of vertices of degrees 1 and 2, respectively. By the routinely procedure, we have
$n_2- \overline{n_2} = 2\gamma - n -2$ and $k - p= 2\gamma - n - n_1+1.$

Now the function $f(k-p)$ can be writen as
 \begin{eqnarray}
\label{43}
f(n_1) = (3/4)^{2n_1} 2^{2(n+2)} 3^{2(-2)}, \text{ for } 3\leq n_1 \leq 3\gamma - n.
\end{eqnarray}

By Lemma 1.1, we have $n_1 \geq 3\gamma - n$. Thus, $n_1 = 3\gamma -n$ is the unique one such that $f(n_1)$ is maximal. Then
there is a corresponding tree with $n$ vertices and domination number $\frac{n+3}{3} \leq \gamma \leq \frac{n}{2}$ such that the vertices in any minimal dominating set $D$ have degrees $1$ and $2$, and the vertices in $\overline{D}$ have degrees $2$ and $3$.

Thus, $f(3\gamma-n) = 4^{3n-6\gamma+2} 9^{3\gamma - n -2}$ is the unique value and is the maximal value of $f(k-p)$ in Case 1. Now that $n_1 = 3\gamma - n$ yields that  $k-p = - \gamma +1$. By the relations (\ref{15}) and (\ref{19}), we have $k=0, p=\gamma-1$ and $l = n-\gamma$.

By the definition of the domination number, a vertex with more
than one pendent neighbor belongs to every minimum dominating set
of a tree, implying that every vertex in a tree T, obtained as
described above, has at most one pendent neighbor. By previous
considerations, the resulting extremal trees, for which equality
holds in Theorem \ref{310}, belong to the graphs in Definition 3.2 (ii).

This completes the proof.
\end{proof}

\subsection{Lower bounds of $\prod_2(G)$ on  trees of domination number $\gamma$}

Let $G$ be a tree of $n$ vertices and domination number $\gamma$. In order to find the minimal values of $\prod_2(G)$,  we first consider the case of $1 \leq \gamma \leq \frac{n}{3}$. The corresponding extreme graphs are given in Definition 2.

\begin{thom}
\label{up11}
Let $G$ be a tree of $n$ vertices and domination number  $1 \leq \gamma \leq \frac{n}{3}$. Then
 $$\prod_2(G) \geq 4^{2\gamma-2} \lfloor{\frac{n-\gamma}{\gamma}}\rfloor^{\lfloor{\frac{n-\gamma}{\gamma}}\rfloor(2\gamma - n + \gamma \lfloor{\frac{n-\gamma}{\gamma}}\rfloor)} (\lfloor{\frac{n-\gamma}{\gamma}}\rfloor +1)^{\lfloor{\frac{n-\gamma}{\gamma}}\rfloor(n-\gamma - \gamma \lfloor{\frac{n-\gamma}{\gamma}}\rfloor)},$$
  where the equality holds if and only if $G \cong D_{n,\gamma}$.
\end{thom}

\begin{proof} We proceed on $g(k-p)$ and determine its minimum.
If $n=3$, then $T \cong P_3$,  $\gamma = 1$ and Theorem \ref{up11} is true. If $n > 3$, as $\gamma \leq \frac{n}{3}$, then $n-\gamma \geq \frac{2n}{3} ~\text{and } \frac{\gamma -1}{n - \gamma} \leq \frac{n-3}{2n} \leq \frac{1}{2}.$
By the relations (\ref{=1}) and (\ref{q>2}), we have
\begin{eqnarray}
\nonumber g(k-p) &&= (q+1)^{(q+1)(n-1+(k-p) - \gamma q)} q^{q(1-n+\gamma - (k-p) + \gamma q)} 2^{2(n-1+(p-k) - (n-\gamma) \cdot 1)} \cdot 1
\\  \nonumber
  && = (q+1)^{(q+1)(k-p)} q^{q(- (k-p) )} 2^{2(-(k-p))} \cdot (q+1)^{(q+1)(n-1- \gamma q)} q^{q(1-n+\gamma + \gamma q)} 2^{2(n-1 - (n- \gamma))}
\\  \nonumber
  && = (\frac{1^1/2^2}{q^q/(q+1)^{(q+1)}})^{(k-p)}
\cdot (q+1)^{(q+1)(n-1- \gamma q)} q^{q(1-n+\gamma + \gamma q)} 2^{2(\gamma -1)}.
\end{eqnarray}
As $q \geq 2, n, \gamma$ are fixed, by Proposition 2, $g(k-p)$ is an increasing function with the variable of $k-p$.
Since $|k-p| \leq \gamma -1$, then there are two cases below.

{\small\bf Case 1:} $0 \leq k-p \leq \gamma -1$.

 By the relations (\ref{23}) and (\ref{24}), $k-p$ falls in two intervals and the minimum values of $g(k-p)$ arrived at either $k-p=0$ or $k-p = \gamma \lfloor{\frac{n-1}{\gamma}}\rfloor + \gamma -n +1$.
In order to find which one is bigger, we need to compare $g(\gamma \lfloor{\frac{n-1}{\gamma}}\rfloor + \gamma -n +1)$ and $g(0)$. Note that
$\gamma \lfloor{\frac{n-1}{\gamma}}\rfloor + \gamma -n +1 \geq (n - 1 -\gamma) + \gamma - n +1 = 0$.
\begin{eqnarray}
\label{255}
\nonumber \frac{g(\gamma \lfloor{\frac{n-1}{\gamma}}\rfloor + \gamma -n +1)}{g(0)} &&= \frac{(1^1/2^2)^{(\gamma \lfloor{\frac{n-1}{\gamma}}\rfloor + \gamma -n +1)}}{(q^q/(q+1)^{(q+1)})^{ 0}}
\\  \nonumber
  && = (1^1/2^2)^{(\gamma \lfloor{\frac{n-1}{\gamma}}\rfloor + \gamma -n +1)}
\\  
  && \geq (1^1/2^2)^{ 0}
= 1.
\end{eqnarray}

Thus, $g(0)$ is minimum when $ 0 \leq k - p \leq \gamma -1$. Also,
\begin{eqnarray}
\label{265}
\nonumber g(0) &&= (q+1)^{(q+1)(n-1-\gamma -q)} q^{q(1-n+\gamma+\gamma q)} 2^{2(n-1- (n- \gamma))}
\\
  && =(\lfloor{\frac{n-1}{\gamma}}\rfloor+1)^{(\lfloor{\frac{n-1}{\gamma}}\rfloor+1)(n-1- \gamma \lfloor{\frac{n-1}{\gamma}}\rfloor)} \lfloor{\frac{n-1}{\gamma}}\rfloor  ^{\lfloor{\frac{n-1}{\gamma}}\rfloor(1-n+\gamma + \gamma \lfloor{\frac{n-1}{\gamma}}\rfloor) 2^{2(\gamma -1)}}.
\end{eqnarray}

{\small\bf Case 2:} $- \gamma +1 \leq k-p \leq 0$.

Note that $\frac{n-1}{\gamma} - 1 \leq \frac{n-\gamma}{ \gamma} \leq \frac{n-1+k-p}{\gamma} \leq \frac{n-1}{\gamma}$.
Let $n-1 =  Q\gamma +R$, where $0 \leq R \leq \gamma -1$. We pay our attention on the case of $0 \leq R \leq \gamma -2$ firstly. Note that $g(k-p)$ is an incresing function on these two intervals of the relations (\ref{27}) and (\ref{28}), for $0 \leq R \leq \gamma-2$. Thus, $g(k-p)$ arrives at the minimum value for either $k-p = \gamma \lfloor{\frac{n-1}{\gamma}}\rfloor - n +1$ or $k-p = - \gamma +1$ (If $R = \gamma -1$, then $\gamma \lfloor{\frac{n-1}{\gamma}}\rfloor - n +1 = -\gamma +1$).

Note that $\gamma \lfloor{\frac{n-1}{\gamma}}\rfloor - n +1 \geq n-1-\gamma-n+1 \geq -\gamma$, $\frac{n-\gamma}{\gamma} = \lfloor\frac{n-1}{\gamma}\rfloor + \frac{R-\gamma+1}{\gamma}$ and $0 \leq R \leq \gamma -2$.
By combing above relations,
\begin{eqnarray}
\label{295}
\nonumber \frac{g(\gamma \lfloor{\frac{n-1}{\gamma}}\rfloor - n +1)}{g(-\gamma +1)}
  && = (\frac{(1^1)/(2^2)}{Q^Q/(Q+1)^{(Q+1)}})^{((\gamma \lfloor{\frac{n-1}{\gamma}}\rfloor - n +1)) - (- \gamma +1))}\\
 && \nonumber = (\frac{1^1/2^2}{Q^Q/(Q+1)^{(Q+1)}})^{((\gamma (\frac{n-\gamma}{\gamma} - \frac{R-\gamma+1}{\gamma}) - n +1)) - (- \gamma +1))}\\
&& \nonumber = (\frac{1^1/2^2}{Q^Q/(Q+1)^{(Q+1)}})^{(\gamma - R -1)}\\
  &&  \geq  (\frac{1^1/2^2}{Q^Q/(Q+1)^{(Q+1)}})^{2 \cdot 1} \geq 1.
\end{eqnarray}

Finally, we need to compare these two minimum values $g(-\gamma+1)$ and $g(0)$. Since $-\gamma+1 \leq 0$ and $g(k-p)$ is increasing, then
 the smallest value of $g(k-p)$ is arrived at $k-p = - \gamma +1$ and for $k=0,$ $p= \gamma-1$ and $l = n - \gamma$. Hence, Theorem \ref{up11} is true and the equality holds if and only if $G \in \mathcal{D}_{n,\gamma}$.
\end{proof}

At the last part, we consider the case of $\frac{n}{3} \leq \gamma \leq \frac{n}{2}$, for a tree $G$ with $n$ vertices and domination number $\gamma$.

\begin{thom}
\label{311}
Let $G$ be a tree of $n$ vertices and domination number  $\frac{n}{3} \leq \gamma \leq \frac{n}{2}$. Then
\begin{eqnarray}
\nonumber
\prod_2(G) \geq
\begin{cases}
4^{n-2}, &\gamma = \lceil{\frac{n}{3}}\rceil,
 \cr 4^{3n-6\gamma+2} 27^{3\gamma - n -2}, &\frac{n+3}{3} \leq \gamma \leq \frac{n}{2}, \end{cases}
\end{eqnarray}
 where the equality holds if and only if $G \cong L_{n,\gamma}$.
\end{thom}

\begin{proof}
We proceed on $g(k-p)$ and determine its minimum.
If $\gamma = \lceil{\frac{n}{3}}\rceil$, $P_n$ is a path. Thus, $\prod_2(G)$ is minimal and Theorem \ref{311} is true. Here we consider the case of $\gamma \geq \frac{n+3}{3}$.

Note that $2\gamma \leq n \leq 3\gamma -3$ yields $\gamma \geq 3$ and $n \geq 6$.
By the relation (\ref{123}), we need to consider two cases below.

{\bf \small Case 1:} $q_1 = \lfloor{\frac{n-1+p-k}{n-\gamma}}\rfloor = 1$.

Note that $1  \leq \frac{n-1+p-k}{n-\gamma} < 2$ yields $p - k < n - 2\gamma +1,$ i.e., $$k-p \geq 2\gamma -n.$$
If $2\gamma - n \leq -1,$ then $2 \leq \frac{n-1}{\gamma} \leq 3 \frac{n-1}{n+3} < 3$. Thus, $\lfloor{\frac{n-1}{\gamma}}\rfloor = 2,$ for $2\gamma - n \leq -1$.
Now we first consider $2\gamma -n \leq k-p \leq 0$. By the same ideas of (\ref{27}) and (\ref{28}), we have that the relations (\ref{32}) and (\ref{33}) hold. Similar to (\ref{34}), we obtain that
\begin{eqnarray}
\label{34110}
\nonumber g(k-p) && = (2+1)^{(2+1)(n-1+k-p-2\gamma)} 2^{2(\gamma - n +1 - (k-p) +2\gamma)} (1+1)^{2(n-1 +(p-k) - (n-\gamma))} 1
\\
  &&  =(27/16)^{2(k-p)} 3^{3(n-1-2\gamma)} 2^{2(4\gamma-n)},~\text{for}~ 2\gamma -n+1 \leq k-p \leq 0.
\end{eqnarray}

Next assume that $0 \leq k-p \leq \gamma -1.$ By the same ideas of the relation (\ref{23}) and (\ref{24}), we have the relations (\ref{35}) and (\ref{36}) hold.
By the relations (\ref{221}), (\ref{35}) and (\ref{36}), we have

 \be
\label{3711}
g(k-p) =  (27/16)^{2(k-p)} 3^{3(n-1-2\gamma)} 2^{2(4\gamma-n)}, ~\text{for}~ 0 \leq k-p \leq 3\gamma  - n,
\ee
and
\be
\label{3811}
g(k-p) =  16^{k-p} 2^{6n- 16\gamma -8}, ~\text{for}~ 3\gamma  - n +1 \leq k-p \leq \gamma - 1.
\ee

Togethering with the relations (\ref{34110}) and (\ref{3711}), we have
\begin{eqnarray}
\label{3911}
 f(k-p)   =(27/16)^{2(k-p)} 3^{3(n-1-2\gamma)} 2^{2(4\gamma-n)},~\text{for}~ 2\gamma -n+1 \leq k-p \leq 3\gamma -n.
\end{eqnarray}

By the relations  (\ref{3811}) and (\ref{3911}), we have  $\frac{f(3\gamma -n)}{f(3\gamma -n+1)} <  1$.
Since the maximal value $f(3\gamma -n)$ of (\ref{3911}) is smaller than the minimum value $f(3\gamma -n+1)$ of the relation (\ref{3811}). In order to find the minimum value of $g(k-p)$, we should consider the relation (\ref{3911}) only.

To find the sharp lower bound of $\prod_2(G)$, where $G$ is a tree with  $n$ vertices and domination n   number $\gamma$. It is enough to find the minimum realizable value of $k-p$, such that the corresponding tree exists. We will proceed on these steps below.

 First, note that an extreme tree $G$ with a minimum $\prod_2(G)$ contains vertices of degree $1, 2$ or $3$.
By the above considerations, any minimal dominating set $D$ has $n_3$ verices of degree $3$ and $n_2$ vertices of degree $2$, i.e., $n_2+n_3 = \gamma$. Also, the set $V(G) \setminus D$ has $n_1$ vertices of degree $1$ and $\overline{n_2}$ vertices of degree $2$, i.e., $\overline{n_2} + n_1 = n- \gamma.$

As $n = n_1+n_2+\overline{n_2}+n_3,$ the relation $(1)$ can be written as $n_1+2(n_2+\overline{n_2}) + 3n_3 = 2(n_1+n_2+\overline{n_2}+n_3) - 2$. Thus,
 \begin{eqnarray}
\label{4000}
n_3 = n_1 -2.
\end{eqnarray}

Combining these relations, we have $n_2 - \overline{n_2} = 2\gamma - n +2$. By using  (\ref{4000}), the relations  (\ref{16}) and  (\ref{17}) could be $n-1+k-p = 2n_2 +3n_1-6$ and $n-1+p-k = 2\overline{n_2}+n_1$. Thus, $k-p= n_1+2\gamma-n-1$.

 Thus,  the function  (\ref{3911}) can be expressed as
  \begin{eqnarray}
\label{4111}
g(n_1) =2^{4n_1+2n-8\gamma -12}, ~\text{for}~ 2 \leq n_1 \leq \gamma+1.
\end{eqnarray}

Now we turn to the case of $2\gamma - n = 0$, i.e., $\gamma = n/2$ if $n$ is even. Then $\lfloor{\frac{n-1}{\gamma}}\rfloor = 1$ and similar to the relations  (\ref{23}) and  (\ref{24}), we have $q = \lfloor{\frac{n-1+k-p}{\gamma}}\rfloor = \lfloor{\frac{n-1}{\gamma}}\rfloor = 1$ for $k-p=0$ and $q = \lfloor{\frac{n-1}{\gamma}}\rfloor +1 =2,$ for $1 \leq k-p \leq \frac{n}{2} -1.$
Recall that at the same time $q_1 = 1$ and, consequently, it has to be $q =2$ (since for $q = 1, T \cong P_n,$ a contadiction, as $\gamma \geq \frac{n+3}{3}$).

By the same method above, we have $g(n_1) =  2^{4n_1+2n-8\gamma -12},$ for $2 \leq n_1 \leq \frac{n}{2}$.
Thus, we should find the minimal value of $n_1$ such that there exists such trees with n vertices and domination number $\gamma$ with
\begin{eqnarray}
\label{42}
\frac{n+3}{3} \leq \gamma \leq \frac{n}{2}.
\end{eqnarray}
 Note that the vertices from any dominating set $D$ of $G$ have degrees 2 and 3, and the vertices $\overline{D}$ have degrees 1 and 2.
By Lemma 1.1, we have $n_1 \geq 3\gamma -n.$ Then the minimal possible value of $f(n_1)$ is achieved for $n_1 = 3\gamma-n$,
 i.e., $k-p = 5\gamma -2n-1$ and $g(5\gamma-2n-1) = 4^{3n-6\gamma+2} 27^{3\gamma - n -2}$. In addition, the extreme graphs of achieving the equality in Theorem \ref{311} belong to $\mathcal{L}(n,\gamma)$.

{\bf \small Case 2:} $q_1 = \lfloor{\frac{n-1+p-k}{n-\gamma}}\rfloor = 2.$

Note that $2 \leq \frac{n-1+p-k}{n-\gamma} < 3$ yields that $p-k \geq n-2\gamma +1$, that is $k-p \leq 2\gamma -n -1 < 0$. Also, $1 \leq \frac{n-1-\gamma+1}{\gamma} \leq \frac{n-1+k-p}{\gamma} \leq \frac{n-1+2\gamma-n-1}{\gamma} = 2 \frac{\gamma -1}{\gamma} < 2$ implies that $q = \lfloor{\frac{n-1+k-p}{\gamma}}\rfloor = 1.$

For $p-k = n - 2\gamma+1$ and any minimal dominating set $D$, $\frac{n-1+p-k}{n-\gamma} = 2$ and all vertices of $\overline{D}$ are degree of 2. If the vertices of $D$ are degree of $1$  or  $2$, then $T \cong P_n$ and it is contradicted with the assumption.
By $\gamma \geq \frac{n+3}{3}$, we have $p-k \geq n-2\gamma+2,$ i.e., $k-p \leq 2\gamma - n -2$.

Thus, we have $$g(k-p) = (16/27)^{(k-p)} 2^{2(3n-4\gamma)} 3^{3(2\gamma - n-1)},~  \text{for}~ -\gamma +1 \leq k-p \leq 2\gamma - n - 2.$$ Next we need to determine the maximal realization of $k-p$ such that the related tree exists. Here we will proceed in by the same method of previous case.
Let $n_1, n_2$ be the number of vertices of degrees 1 and 2, respectively. By the routinely procedure, we have
$n_2- \overline{n_2} = 2\gamma - n -2$ and $k - p= 2\gamma - n - n_1+1.$

Now the function $g(k-p)$ can be writen as
 \begin{eqnarray}
\label{4311}
g(n_1) = (27/16)^{n_1} 2^{2(n+2)} 3^{2(-3)}, \text{ for } 3\leq n_1 \leq 3\gamma - n.
\end{eqnarray}

By Lemma 1.1, we have $n_1 \geq 3\gamma - n$. Since $g(n_1)$ is an increasing function, then $n_1 = 3\gamma -n$ is the unique one such that $g(n_1)$ is minimal. Then
there is a corresponding tree with $n$ vertices and domination number $\frac{n+3}{3} \leq \gamma \leq \frac{n}{2}$ such that the vertices in any minimal dominating set $D$ have degrees $1$ and $2$, and the vertices in $\overline{D}$ have degrees $2$ and $3$.

Thus, $f(3\gamma-n) = 4^{3n-6\gamma+2} 27^{3\gamma - n -2}$ is the unique value and is the maximal value of $f(k-p)$ in Case 1. Now that $n_1 = 3\gamma - n$ yields that  $k-p = - \gamma +1$. By the relations (\ref{15}) and (\ref{19}), we have $k=0, p=\gamma-1$ and $l = n-\gamma$.

By the definition of the domination number, a vertex with more
than one pendent neighbor belongs to every minimum dominating set
of a tree, implying that every vertex in a tree T, obtained as
described above, has at most one pendent neighbor. By previous
considerations, the resulting extremal trees, for which equality
holds in \ref{311}, belong to the graphs in Definition 3.2 (ii).

This completes the proof.
\end{proof}

\vskip4mm\noindent{\bf Acknowledgements.}


 The work was partially supported by the
National Science Foundation of China under Grant nos.  11271149, 11371162 and 11601006, the Natural
Science Foundation for the Higher Education Institutions of Anhui
Province of China under Grant no. KJ2015A331. Also, it was
partially supported by the Self-determined Research Funds of CCNU
from the colleges basic research and operation of MOE.

\end{document}